\documentclass{amsart}

\usepackage[latin1]{inputenc}
\usepackage[all]{xy}
\usepackage[english]{babel}
\usepackage{amssymb}
\usepackage[lite, initials]{amsrefs}

\newtheorem{teorema}{Theorem}[section]
\newtheorem{lemma}[teorema]{Lemma}
\newtheorem{propos}[teorema]{Proposition}

\theoremstyle{definition}

\newtheorem{defin}[teorema]{Definition}

\newcommand{\Ol}{\mathcal{O}}
\newcommand{\C}{\mathbb{C}}
\newcommand{\N}{\mathbb{N}}
\newcommand{\R}{\mathbb{R}}

\title{Union problem for holomorphically convex spaces}
\author[S.~Mongodi]{Samuele Mongodi\textsuperscript{1}}
\address{\textsuperscript{1}Politecnico di Milano, Dipartimento di Matematica, Via Bonardi, 9 -- I-20133 Milano, Italy}
\email{samuele.mongodi@polimi.it}
 \date{\today}
\begin{document}

\maketitle

\section{Introduction}

In this short note, we collect some results regarding the Remmert reduction of holomorphically convex space and its application to a variation of the usual union problem. Classically, the union problem asks the following question: is a complex space, which is an increasing union of Stein subspaces $X_1\Subset X_2\Subset\cdots$, a Stein space itself?

Many partial results were obtained. Many results give positive answers, for instance
\begin{itemize}
\item when $X$ is contained in $\C^n$ (\cite{BehSte})
\item when $X$ is contained in an unramified Riemann domain over $\C^n$ (\cite{GraRem2})
\item when $X$ is a reduced complex space and every pari $(X_j, X_{j+1})$ is Runge (\cite{Ste})
\item when $X$ is a reduced complex space and $H^1(X,\Ol)=0$ (\cites{Mar, Sil})
\item when $X\subset S$, with $S$ Stein, each $X_j$ is open in $S$ and $\dim H^1(X,\Ol)<+\infty$ (\cite{Tov})
\item when $X\subset S$, with $S$ a Stein manifold and each $X_j$ is open in $S$ (\cite{DocGra}, where actually more is proved).
\end{itemize}
We have also negative answers, as shown by the famous construction by Fornaess in \cite{For}.

The variation we are interested in is the following: is a complex space, which is an increasing union of holomorphically convex subspaces $X_1\Subset X_2\Subset\cdots$, holomorphically convex itself? The results presented here are close analogues of (some of) those listed above; our aim is only to collect such material for reference, as we consider it well known.

\section{Remmert reduction}

Let $(X,\Ol_X)$ be a complex analytic space (given as a locally ringed space); we recall that a function $f:X\to Y$ is holomorphic if and only if the induced map $(f,f_*):(X,\Ol_X)\to (Y,\Ol_Y)$ is a morphism of locally ringed spaces, i.e. if there is a morphism of sheaves of rings $T:\Ol_Y\to f_*\Ol_X$ which makes $f_*\Ol_X$ a module over $\Ol_Y$. In particular, if $f_*\Ol_X$ and $\Ol_Y$ coincide as sheaves of rings, the map is holomorphic.

In what follows, if we write an equality between sheaves, we always mean that there exists an isomorphism in the correct category (usually the one of sheaves of rings).

\medskip

\begin{defin}Suppose that $X$ is holomorphically convex, then we call a \emph{Remmert reduction} of $X$ a pair $(\phi, Y)$ where $Y$ is a Stein space and $\phi:X\to Y$ is a proper holomorphic map such that $\phi_*\Ol_X=\Ol_Y$ (as sheaves).
\end{defin}
\medskip

We can construct, following Cartan \cite{Car}, a Remmert reduction of $X$ as follows:
we say that $x,y\in X$ are equivalent if $f(x)=f(y)$ for all $f\in\Ol_X(X)$ and we call this relation $\sim$, then, by a theorem of Cartan, $Y=X/\sim$ is a complex analytic space and, if we call $\phi:X\to Y$ the quotient map, we have that
$$\phi_*\Ol_X=\Ol_Y$$
as sheaves, which implies that $\phi$ is a morphism of locally ringed spaces, hence, in this case, a holomorphic map.

\medskip

\begin{propos}\label{fatto1}Suppose $Z$ is a complex analytic space and $\psi:X\to Z$ is holomorphic and proper, then $\psi_*\Ol_X=\Ol_Z$ if and only if $\psi$ is surjective and has connected fibers.\end{propos}
\begin{proof} If $\psi$ is proper, surjective and has connected fibers, then every function $f\in\Ol_X(X)$ is constant on each fiber of $\psi$, which means that $f$ can be pushed forward to a holomorphic function on $Z$, hence $\psi_*\Ol_X=\Ol_Z$.

On the other hand, if $\psi_*\Ol_X=\Ol_Z$, consider the Stein factorization of $\psi:X\to Z$, by $f:X\to W$ and $g:W\to Z$, where $W$ is an analytic space, $f,g$ are holomorphic maps, $f_*\Ol_X=\Ol_W$ and has connected fibers and $g$ is finite. As $\psi_*\Ol_X=\Ol_Z$, $\psi=g\circ f$ and $f_*\Ol_X=\Ol_W$, then $g_*\Ol_W=\Ol_Z$. But $g$ is a finite map, so $g_*\Ol_W$ has rank $1$ over $\Ol_Z$ if and only if $g$ is an isomorphism. So, $\psi$ is surjective and its fibers are connected. \end{proof}

\medskip

\begin{teorema}\label{fatto2}All the Remmert reductions of $X$ are isomorphic\end{teorema}
\begin{proof} If $(\phi, Y)$ is the Remmert reduction given by the Cartan construction and $(\phi', Y')$ is another, then we can construct a bijective map $a:Y\to Y'$. Indeed, consider a point $x\in X$ and the sets $\phi^{-1}(\phi(x))=A_x$ and $(\phi')^{-1}(\phi'(x))=B_x$; both sets are compact (by properness) and connected (by Proposition \ref{fatto1}); therefore, every holomorphic function $f\in\Ol_X(X)$ will be constant on $B_x$, hence $B_x\subseteq A_x$. Take $x_1\in X\setminus B_x$ and consider $y_0=\phi'(x)$ and $y_1=\phi'(x_1)$, which are two distinct points in $Y'$; as $Y'$ is Stein, there exists $g\in\Ol_{Y'}(Y')$ such that $g(y_0)\neq g(y_1)$ and, as $\phi'_*\Ol_X=\Ol_{Y'}$, there exists $h\in\Ol_X$ such that $h=g\circ\phi'$, so $h(x)\neq h(x_1)$, so $x_1\not\in A_x$. Therefore $A_x=B_x$.

Hence, we can define $a:Y\to Y'$ bijective that sends $y$ to $y'$ if $\phi^{-1}(y)=(\phi')^{-1}(y')$. Moreover, given $U$ open in $Y'$ and $g\in\Ol_{Y'}(U)$, we can find $f\in\Ol_{X}((\phi')^{-1}(U))$ such that $\phi'_*f=g$; therefore $\phi_*f=h\in\Ol_Y(\phi((\phi')^{-1}(U)))$ and $\phi((\phi')^{-1}(U))=a^{-1}(U)$ and $a_*g=h$, by construction. So $a_*\Ol_Y=\Ol_{Y'}$, so $a$ is holomorphic. It is easy to swap  $Y$ and $Y'$ and carry on the same argument on $a^{-1}$ showing it is holomorphic as well.

We have constructed a map $a:Y\to Y'$ which is biholomorphic; to conclude it is enough to show that, by construction $\phi'=a\circ\phi$, so $a$ is an isomorphism between $(\phi, Y)$ and $(\phi', Y')$. \end{proof}

Given $X$ holomorphically convex and $x\in X$, define $S_x$ as the union of all the compact connected complex analytic subspaces of $X$ that contain $x$.

\begin{lemma}\label{fatto3} If $X$ ia holomorphically convex, $S_x$ is compact for every $x$.\end{lemma}
\begin{proof} If $S_x$ is non compact for some $x\in X$, then there is a sequence of points $x_n\in S_x$ which escape every compact. Therefore, by holomorphic convexity, there exists a function $f\in\Ol_X(X)$ such that $f(x_n)=n$.

On the other hand
$$S_x=\bigcup_{i\in I} Z_i$$
where $Z_i$ is a compact connected analytic subspace of $X$ and $x\in Z_i$. Therefore, $f\vert_{Z_i}$ is constant and equal to $f(x)$, so $f$ is constant on $S_x$, a contradiction. Therefore $S_x$ is compact. \end{proof}

Given $x,x'\in X$, say that $xRx'$ if $S_x\cap S_{x'}\neq \emptyset$.

\medskip

\begin{lemma}\label{fatto4}$R$ is a proper relation.\end{lemma}
\begin{proof} Take $K\Subset X$ and consider
$$K'=\bigcup_{x\in K}S_x\;.$$
As $X$ is holomorphically convex, if every holomorphic function is bounded on $K'$, then $K'$ is compact; so, consider $f\in\Ol_X(X)$. Obviously, $f$ is bounded on $K$ and $f\vert_{S_x}\equiv f(x)$, which implies that $f\vert_{K'}$ is bounded. \end{proof}
\medskip

\begin{propos}\label{fatto5}Define $\pi:X\to X/R$ and $\Ol_{X/R}$ in the usual way, then the global sections of $\Ol_{X/R}$ separate points in $X/R$.\end{propos}
\begin{proof}Take $z_1,\  z_2\in X/R$ and $x_1,\ x_2\in X$ such that $\pi(x_j)=z_j$ for $j=1,2$. If every holomorphic function $f\in\Ol_X(X)$ is such that $f(x_1)=f(x_2)$, then $\phi(x_1)=\phi(x_2)$, where $(\phi, Y)$ is the Remmert reduction of $X$, so $x_1$ and $x_2$ are contained in the same fiber of $\phi$, which is a compact connected analytic space, so $S_{x_1}\cap S_{x_2}\neq\emptyset$, so $x_1Rx_2$, so $\pi(x_1)=\pi(x_2)$, so $z_1=z_2$.

So, if $z_1\neq z_2$, we can find $f\in\Ol_X(X)$ such that $f(x_1)\neq f(x_2)$. Now, $\pi_*f\in \Ol_{X/R}(X/R)$ separates $z_1$ and $z_2$. \end{proof}

By Lemma \ref{fatto4} and Proposition \ref{fatto5}, we can apply Cartan's result on quotients by analytic relations (see \cite{Car}) and obtain that $(X/R, \Ol_{X/R})$ is a complex analytic space. Moreover, $\pi_*\Ol_X=\Ol_{X/R}$ by definition, so $\pi:X\to X/R$ is a proper, surjective, holomorphic map. Also, $X/R$ is holomorphically convex, as $X$ is, and elements of $\Ol_{X/R}(X/R)$ separate points, so $X/R$ is Stein. Therefore, $(\pi, X/R)$ is the Remmert reduction of $X$, so it is isomorphic to $(\phi, Y)$.

\medskip

\begin{teorema}[Universal property of the Remmert reduction]\label{fatto6} If $(\pi, Y)$ is the Remmert reduction of $X$, $S$ is a Stein space and $f:X\to S$ is a holomorphic map, then there exists $g:Y\to S$ such that $g\circ\pi=f$.\end{teorema}
\begin{proof} There is a proper holomorphic embedding $F:S\to\C^N$, so we consider the map $F\circ f=(h_1,\ldots, h_N)$ with $h_j\in\Ol(X)$. Therefore, we have $g_j\in\Ol(Y)$ such that $h_j=g_j\circ \pi$, for $j=1,\ldots, N$. Let $G:F(S)\to S$ be the inverse of $F$ on $F(S)$, then $g=G\circ(g_1,\ldots, g_N)$ is the required map. \end{proof}

\section{Runge pairs}

Let $X_1\subset X_2\subset X_3\subset\ldots\subset X_n\subset X_{n+1}\subset\ldots\subset X$ be a sequence of increasing domains in a complex space $X$ and suppose that $X_j$ is relatively compact and Runge in $X_{j+1}$ for every $j$, i.e. the restriction map $r:\Ol_{X_{j+1}}(X_{j+1})\to\Ol_{X_j}(X_j)$ has dense image, i.e. for every $K\Subset X_j$ and every $f:X_j\to\C$ holomorphic, we can find a sequence $\{f_k\}_{k\in\N}$ of holomorphic functions on $X_{j+1}$ such that
$$\sup_{x\in K}|f_k(x)-f(x)|\to 0\qquad \textrm{as}\ \ k\to\infty\;.$$

\medskip

\begin{lemma}\label{fatto7}Let $X$ be a complex space and $X'\subset X$ an open domain. Suppose that $X$ and $X'$ are both holomorphically convex and that $X'$ is Runge in $X$, let $(\pi', Y')$ and $(\pi,Y)$ be the Remmert reductions of $X'$ and $X$, then $Y'$ embeds in $Y$ as an open domain.\end{lemma}
\begin{proof}Let $j:X'\to X$ be the inclusion and consider the map $f:X'\to Y$ given by $f=\pi\circ j$. By the universal property of the Remmert reduction, as $Y$ is Stein, we can factor $f$ as $\sigma\circ \pi'$, with $\sigma:Y'\to Y$ holomorphic.

Now, suppose that we have two points $u,\ v\in Y'$ such that $\sigma(u)=\sigma(v)$, then we have two points $z,w\in X'$ such that $\pi'(z)\neq\pi'(w)$ but $f(z)=f(w)$, therefore, as $j$ is injective, $j(z)=a$, $j(w)=b$, $a,b\in X$ and $a\neq b$, with $\pi(a)=\pi(b)$. As $Y'$ is Stein, there exists a function $\phi:Y'\to\C$ which separates $u$ from $v$, hence $\psi=\phi\circ \pi'$ is a holomorphic function on $X'$ which has different values in $z$, $w$. By the Runge property, we can find a holomorphic function $h:X\to\C$ which approximates $\psi$ as well as we want on $X'$, therefore, we can find such an $h$ with different values in $a$, $b$, but this contraddicts the fact that $\pi(a)=\pi(b)$.

Therefore $\sigma$ is injective. Moreover, let $K$ be a compact subset of $Y$ which is contained in $\sigma(Y')$; if $\sigma^{-1}(K)$ is not compact, we can find a sequence $\{y_n\}\subset \sigma^{-1}(K)$ that escapes every compact of $Y'$, so, as $Y'$ is Stein, we have a holomorphic function $f:Y'\to\C$ such that $|f(y_n)|\to+\infty$.

The set $L=\pi^{-1}(K)$ is compact in $X$ and is contained in $j(X')$. By the Runge property, we can approximate $f\circ \pi'$ on $K$ with a function $g\in\Ol_X(X)$, so $|f\circ\pi' - g|<1$ on $L$, so $g$ is unbounded on $L$, hence $L$ is not compact, which is a contradiction. Hence $\sigma^{-1}(K)$ is compact in $Y'$.

So $\sigma$ is an embedding of $Y'$ into $Y$. \end{proof}

\medskip

\begin{teorema}[Oka-Weil for holomorphically convex spaces]\label{fatto8} Let $S$ be a holomorphically convex space and $K\Subset S$ a $\Ol(S)$-convex compact set, then every function holomorphic in a neighborhood of $K$ can be approximated uniformly on $K$ by functions in $\Ol(S)$.\end{teorema}
\begin{proof} Let $\pi:S\to T$ be the Remmert reduction of $S$; it is easy to see that $K$ is $\Ol(S)$-convex if and only if $K=\pi^{-1}(\pi(K))$ and $\pi(K)$ is $\Ol(T)$-convex. Given $U$ a neighborhood of $K$, we can find an open set $V\subset U$ such that $K\Subset V$ and $\pi^{-1}(\pi(V))=V$, so, given $f\in\Ol_S(U)$, we can restrict it to $V$ and notice that, by the alternative description we gave of the Remmert reduction, $f=g\circ p$ for some $g\in\Ol_T(\pi(V))$.

Now, we can apply the usual Oka-Weil theorem to the compact set $\pi(K)$ in the Stein space $T$, for the holomorphic function $g$; we obtain a sequence of functions $g_j\in\Ol(T)$ such that $g_j\to g$ uniformly on $K$. Define $f_j=g_j\circ \pi$, then it is easy to show that $f_j\in\Ol(S)$ and $f_j\to f$ uniformly on $K$. \end{proof}

\medskip

\begin{teorema}\label{fatto9}If each $X_j$ is holomorphically convex, then $Y=\bigcup X_j$ is holomorphically convex.\end{teorema}
\begin{proof}Take a sequence $\{x_k\}\subset Y$ such that for every compact set $L\subset Y$, there are infinitely many elements of the sequence outside $L$; it is not restrictive to suppose that $x_j\in X_{j}\setminus X_{j-1}$, where $X_{-1}=\emptyset$. Let $\pi_j:X_j\to Y_j$ be the Remmert reduction, by hypothesis and by Lemma \ref{fatto7}, for every $j$, the points $\pi_j(x_1),\ldots, \pi_j(x_j)$ are all different.

Now, for each $j\geq 1$ select a compact $K_j\Subset X_j$ such that
\begin{enumerate}
\item $x_k\in K_j$ for $k\leq j$
\item $K_j$ is $\Ol(X_j)$-convex
\item $X_{j-1}\subset K_j$.
\end{enumerate}

Set $f_1:X_1\to\C$ to be the constant function $1$; suppose we have defined $f_1,\ldots, f_{j}$, then we apply Theorem \ref{fatto8} to obtain $f_{j+1}:X_{j+1}\to\C$ holomorphic such that $|f_{j+1}(x)-f_{j}(x)|<2^{-j-1}$ for $x\in K_j$ and $|f_{j+1}(x_{j+1})-(j+1)|<2^{-j-1}$ (as $K_j\Subset X_j$ and $x_{j+1}\in X_{j+1}\setminus X_j$, the set $K_j\cup \pi^{-1}(\pi(x_{j+1}))$ is $\Ol(X_{j+1})$-convex and $\pi^{-1}(\pi(x_{j+1}))$ is one of its connected components).

Therefore, we have a sequence of holomorphic functions $f_j:X_j\to\C$ such that
$$k-1\leq|f_j(x_k)|\leq k+1\qquad k=1,\ldots, j$$
and, if $j>k$, then
$$\|f_{j+1}-f_j\|_{X_k}=\sup_{x\in X_k}|f_{j+1}(x)-f_{j}(x)|<2^{-j-1}\;.$$

Now, choose $k>0$ integer and consider the sequence $\{f_j\vert_{\overline{X_k}}\}_{j>k}$; we have that
$$\|f_{j+h}-f_{j}\|_{X_k}\leq \sum_{m=1}^h\|f_{j+m}-f_{j+m-1}\|_{X_k}\leq \sum_{m=1}^h2^{-j-m}\leq 2^{-j}$$
therefore the sequence converges uniformly on $\overline{X_k}$. Then, we can define $f=\lim f_j$ as a holomorphic function $f:Y\to \C$.

Moreover, $k-1\leq f(x_k)\leq k+1$ for all $k>0$, so $Y$ is holomorphically convex. \end{proof}

As an application, we have the following.
\begin{teorema}Suppose $M$ is a manifold, endowed with a plurisubharmonic exhaustion $\phi:M\to\R$, and suppose there exists a sequence of real numbers $c_j\to+\infty$ such that
$$X_j=\{x\in M\ :\ \phi(x)\leq c_j\}$$
has a smooth, strictly pseudoconvex boundary (i.e. $\phi$ is smooth and strictly plurisubharmonic in a neighborhood of $bX_j$) for each $j$. Then $M$ is a modification of a Stein space along at most countably many points.\end{teorema}
\begin{proof}Each $X_j$ is holomorphically convex  and, actually, a modification of a Stein space, by \cite{grau1}*{Theorem 1}.

Let us now consider the Remmert reduction $\pi_j:X_j\to Y_j$. As we saw, $Y_j$ is obtained as a quotient with respect to the relation $R$ defined in the first section, so $\pi_j^{-1}(y)$ is a compact complex subspace of $X_j$ for every $y\in Y_j$, therefore $\phi$ is constant on $\pi^{-1}_j(y)$ for every $y\in Y_j$. We define $\psi_j:Y_j\to \R$ by setting $\psi_j(y)=\phi(\pi_j^{-1}(y))$. This functions is clearly continuous; set
$$S_j=\{y\in Y_{j}\ :\ \dim \pi_j^{-1}(y)>0\}\;,$$
then $\pi_j$ is a biholomorphism from $X_j\setminus\pi_j^{-1}(S_j)\to Y_j\setminus S_j$. This implies that $\psi_j$ is plurisubharmonic on $Y_j\setminus S_j$; we know that $S_j$ is a proper complex subspace of $Y_j$ (because $X_j$ is a modification of $Y_j$) and that $\psi_j$ is bounded near every point of $S_j$, then by \cite{GraRem}*{Satz 3} $\psi_j\vert_{Y_j\setminus S_j}$ extends uniquely to a plurisubharmonic function on $Y_j$. Let us call such extension $\tilde{\psi}_j$.

Let $y\in S_j$ and let $S_y=\pi_j^{-1}(y)$. We know that $\phi$ and $\tilde{\psi}_j\circ\pi_j$ are both constant on $S_y$; let $j:\Delta\to M$ be a complex disc intersecting $S_y$ at a single regular point $p\in S_y$, such that $j(0)=p$. The functions $\phi\circ j$ and $\tilde{\psi}_j\circ\pi_j\circ j$ are both subharmonic on $\Delta$ and coincide in $\Delta\setminus\{0\}$, but then
$$\tilde{\psi}_j\circ\pi_j\circ j(0)=\limsup_{\zeta\to 0} \tilde{\psi}_j\circ\pi_j\circ j(\zeta)=\limsup_{\zeta\to 0}\phi\circ j(\zeta)=\rho_{j-1}\circ j(0)\;.$$
So $\phi(p)=\tilde{\psi}_j\circ\pi_{j}(p)$, therefore $\tilde{\psi}_j(y)=\psi_j(y)$, thus proving that $\psi_j$ is (continous and) plurisubharmonic on $Y_j$.

Therefore, by \cite{Nar}*{Corollary 1, Section 4}, $Y_{j-1}=\{\psi_j<c_{j-1}\}$ is Runge in $Y_j$.

This easily implies, by \cite{KK}*{Lemma 63.4}, that $X_{j-1}$ is Runge in $X_j$. We are now in position to use Theorem \ref{fatto9}, in order to prove that $M$ is holomorphically convex.

Finally, let $\pi:M\to Y$ be the Remmert reduction of $M$; as there exists $\Omega\subset M$ such that $\phi$ is strictly plurisubharmonic on $\Omega$,  by the same reasoning from above, $\pi\vert_{\Omega}:\Omega\to \pi(\Omega)$ is biholomorphic. Therefore $Y$ is an irreducible (as $M$ is a connected manifold) Stein space of the same dimension of $M$, then $M$ is a modification of a Stein space.

Moreover, let
$$S=\{y\in Y\ :\ \dim \pi^{-1}(y)>0\}\;.$$
Then $S$ cannot intersect any $bY_j$, for $j>0$, because every $Y_j$ is strictly pseudoconvex; therefore the sets
$$S^j=S\cap\{y\in Y : c_{j-1}<\phi(\pi^{-1}(y))<c_{j}\}$$
are connected components of $S$ and $S^j\Subset Y_{j}$. As $Y_j$ is Stein, $S^j$ consists a finite number of points, so $M$ is the modification of the Stein space $Y$ along $S$, a collection of at most countably many points.\end{proof}


\begin{bibdiv}
\begin{biblist}
\bib{BehSte}{article}{
   author={Behnke, H.},
   author={Stein, K.},
   title={Konvergente Folgen von Regularit\"{a}tsbereichen und die
   Meromorphiekonvexit\"{a}t},
   language={German},
   journal={Math. Ann.},
   volume={116},
   date={1939},
   number={1},
   pages={204--216},
   issn={0025-5831},
   review={\MR{1513225}},
   doi={10.1007/BF01597355},
}
\bib{Car}{article}{
   author={Cartan, Henri},
   title={Quotients of complex analytic spaces},
   conference={
      title={Contributions to function theory},
      address={Internat. Colloq. Function Theory, Bombay},
      date={1960},
   },
   book={
      publisher={Tata Institute of Fundamental Research, Bombay},
   },
   date={1960},
   pages={1--15},
   review={\MR{0139769}},
}
\bib{DocGra}{article}{
   author={Docquier, Ferdinand},
   author={Grauert, Hans},
   title={Levisches Problem und Rungescher Satz f\"{u}r Teilgebiete Steinscher
   Mannigfaltigkeiten},
   language={German},
   journal={Math. Ann.},
   volume={140},
   date={1960},
   pages={94--123},
   issn={0025-5831},
   review={\MR{0148939}},
   doi={10.1007/BF01360084},
}

\bib{For}{article}{
   author={Fornaess, John Erik},
   title={An increasing sequence of Stein manifolds whose limit is not
   Stein},
   journal={Math. Ann.},
   volume={223},
   date={1976},
   number={3},
   pages={275--277},
   issn={0025-5831},
   review={\MR{0417448}},
   doi={10.1007/BF01360958},
}
\bib{grau1}{article}{
   author={Grauert, Hans},
   title={On Levi's problem and the imbedding of real-analytic manifolds},
   journal={Ann. of Math. (2)},
   volume={68},
   date={1958},
   pages={460--472},
   issn={0003-486X},
   review={\MR{0098847}},
   doi={10.2307/1970257},
}
\bib{GraRem2}{article}{
   author={Grauert, Hans},
   author={Remmert, Reinhold},
   title={Konvexit\"{a}t in der komplexen Analysis. Nicht-holomorph-konvexe
   Holomorphiegebiete und Anwendungen auf die Abbildungstheorie},
   language={German},
   journal={Comment. Math. Helv.},
   volume={31},
   date={1956},
   pages={152--160, 161--183},
   issn={0010-2571},
   review={\MR{0088028}},
   doi={10.1007/BF02564357},
}
\bib{GraRem}{article}{
   author={Grauert, Hans},
   author={Remmert, Reinhold},
   title={Plurisubharmonische Funktionen in komplexen R\"{a}umen},
   language={German},
   journal={Math. Z.},
   volume={65},
   date={1956},
   pages={175--194},
   issn={0025-5874},
   review={\MR{0081960}},
   doi={10.1007/BF01473877},
}
\bib{KK}{book}{
   author={Kaup, Ludger},
   author={Kaup, Burchard},
   title={Holomorphic functions of several variables},
   series={De Gruyter Studies in Mathematics},
   volume={3},
   note={An introduction to the fundamental theory;
   With the assistance of Gottfried Barthel;
   Translated from the German by Michael Bridgland},
   publisher={Walter de Gruyter \& Co., Berlin},
   date={1983},
   pages={xv+349},
   isbn={3-11-004150-2},
   review={\MR{716497}},
   doi={10.1515/9783110838350},
}
\bib{Mar}{article}{
   author={Markoe, Andrew},
   title={Runge families and inductive limits of Stein spaces},
   language={English, with French summary},
   journal={Ann. Inst. Fourier (Grenoble)},
   volume={27},
   date={1977},
   number={3},
   pages={vi, 117--127},
   issn={0373-0956},
   review={\MR{0590069}},
}
\bib{Nar}{article}{
   author={Narasimhan, Raghavan},
   title={The Levi problem for complex spaces. II},
   journal={Math. Ann.},
   volume={146},
   date={1962},
   pages={195--216},
   issn={0025-5831},
   review={\MR{0182747}},
   doi={10.1007/BF01470950},
}
\bib{Sil}{article}{
   author={Silva, Alessandro},
   title={Rungescher Satz and a condition for Steinness for the limit of an
   increasing sequence of Stein spaces},
   language={English, with French summary},
   journal={Ann. Inst. Fourier (Grenoble)},
   volume={28},
   date={1978},
   number={2},
   pages={vi, 187--200},
   issn={0373-0956},
   review={\MR{0508090}},
}

\bib{Ste}{article}{
   author={Stein, Karl},
   title={\"{U}berlagerungen holomorph-vollst\"{a}ndiger komplexer R\"{a}ume},
   language={German},
   journal={Arch. Math. (Basel)},
   volume={7},
   date={1956},
   pages={354--361},
   issn={0003-889X},
   review={\MR{0084836}},
   doi={10.1007/BF01900686},
}
\bib{Tov}{article}{
   author={Tovar, L. M.},
   title={Open Stein subsets and domains of holomorphy in complex spaces},
   conference={
      title={Topics in several complex variables},
      address={Mexico},
      date={1983},
   },
   book={
      series={Res. Notes in Math.},
      volume={112},
      publisher={Pitman, Boston, MA},
   },
   date={1985},
   pages={183--189},
   review={\MR{805603}},
}
  \end{biblist}
\end{bibdiv}
\end{document}